\documentclass[11pt]{amsart}
\usepackage{geometry,amsthm,graphicx,float,enumerate}
\usepackage{amsmath}
\usepackage[makeroom]{cancel}
\usepackage{scalefnt}
\usepackage{mathrsfs}
\usepackage{color}

\usepackage{amssymb}
\restylefloat{table}
\restylefloat{figure}

\swapnumbers
\newtheorem{thm}{Theorem}[section]

\newtheorem{cor}[thm]{Corollary}

\newtheorem{lemma}[thm]{Lemma}
\newtheorem{prop}[thm]{Proposition}

\theoremstyle{definition}
\newtheorem{defn}[thm]{Definition}
\theoremstyle{remark}

\newtheorem{ex}[thm]{Example}

\newcommand{\BB}[1]{\mathbb{#1}}
\newcommand{\CAL}[1]{\mathcal{#1}}

\newcommand{\ZM}{[\![m]\!]}
\newcommand{\ZN}{[\![ n ]\!]}
\newcommand{\ZNprime}{[\![n']\!]}
\newcommand{\ZCT}{[\![{c_r}]\!]}

\newcommand{\DF}{d_{\BB{F}}(m)}

\newcommand{\KF}{k_{\BB{F}}(m)}

\newcommand{\TR}{\mathop{\mathrm{tr}}}
\newcommand{\SPAN}{\text{span}}

\newcommand{\BLOCKSET}{\BB{S}}

\newcommand{\CBRACK}[1]{ \left\{ #1 \right\} } 
\newcommand{\SBRACK}[1]{ \left[ #1 \right] } 
\newcommand{\PARENTH}[1]{ \left( #1 \right)} 

\makeatletter
\providecommand*{\cupdot}{%
  \mathbin{%
    \mathpalette\@cupdot{}%
  }%
}

\newcommand*{\@cupdot}[2]{%
  \ooalign{%
    $\m@th#1\cup$\cr
    \sbox0{$#1\cup$}%
    \dimen@=\ht0 %
    \sbox0{$#1\cdot$}%
    \advance\dimen@ by -\ht0 %
    \dimen@=.5\dimen@
    \hidewidth\raise\dimen@\hbox{$\m@th#1\cdot$}\hidewidth
  }%
}

\begin{document}

\title[Maximal Orthoplectic Fusion Frames from Mutually Unbiased Bases and Block Designs$\,\,\,$]{Maximal Orthoplectic Fusion Frames\\ from Mutually Unbiased Bases and Block Designs}

\author{Bernhard G. Bodmann}\thanks{B. G. B. was supported in part by NSF DMS 1412524, J. I. H.  by NSF ATD 1321779.}
\address{651 Philip G. Hoffman Hall, Department of Mathematics, University of Houston, Houston, TX 77204-3008}

\author{John I. Haas}
\address{219 Mathematical Sciences Building, Department of Mathematics, University of Missouri, Columbia, MO 65211}

\begin{abstract}The construction of optimal line packings in real or complex Euclidean spaces has shown to be a tantalizingly difficult task, because it includes the problem of finding
maximal sets of equiangular lines. In the regime where equiangular
lines are not possible, some optimal packings are known, for example, those achieving the orthoplex bound related to maximal sets of mutually unbiased bases. 
In this paper, we investigate the packing of subspaces instead of lines and determine the implications of maximality in this context.
We leverage the existence of real or complex maximal mutually unbiased bases with a combinatorial design strategy in order to 
find optimal subspace packings that achieve the orthoplex bound. We also show that maximal sets of mutually unbiased bases convert between coordinate projections associated with certain balanced incomplete block designs
and Grassmannian 2-designs.
Examples of maximal orthoplectic fusion frames
already appeared in the works by Shor, Sloane and by Zauner. They are realized in dimensions that are a power of four in the real case or a power of two in the complex case. 
%
%
\end{abstract}

\maketitle

\section{Introduction}

The problem of finding the best packings of lines, one-dimensional subspaces of a real or complex Euclidean space, is easy to state. Despite its simple geometric formulation, it has given rise to a surprisingly diverse literature over many years,
ranging from relatively elementary, low dimensional examples \cite{Haantjes1948} to more sophisticated constructions \cite{ConwayHardinSloane1996}, some 
involving combinatorial \cite{vanLintSeidel1966,LemmensSeidel1973,Neumaier1989} or group-theoretic aspects \cite{CalderbankCameronKantorSeidel1997,XiaZhouGiannakis2005}
and results on bounds on the relationship between the number of lines and achievable angles \cite{Rankin1955,Welch1974,Koornwinder1976}. Maximal sets of
equiangular lines are known to be optimal packings, but the number of lines that can be realized is hard to determine \cite{GreavesKoolenMunemasaSzollosi2016}. 
Special regard has been given to the construction of complex examples, motivated by applications in quantum information theory \cite{Zauner1999}.
Numerical searches indicate that they exist in many cases \cite{SustikTroppDhillonHeath2007,ScottGrassl2010}, 
but a rigorous proof of their existence is restricted to low dimensions, see \cite{ScottGrassl2010} and references therein.

Next to lines, packings of higher-dimensional subspaces have also been investigated \cite{ConwayHardinSloane1996,CalderbankHardinRainsShorSloane1999}.  In this case, even less seems to be known about general construction principles that realize tight bounds
 \cite{BargNogin2002,Bachoc2004,Henkel2005,Bachoc2006}.
More recently, these packing problems have been studied in the context of frame theory. Apart from geometric optimality criteria, frame design aims at 
tightness, which implies that the projections onto the subspaces sum to a multiple of the identity.
The case of higher-dimensional subspaces corresponds to
fusion frames. If the number of subspaces is not too large, then in close similarity to line packings, equi-distant fusion frames
present optimal solutions \cite{KutyniokPezeshkiCalderbankLiu2009}. Examples of such constructions follow similar strategies as in the frame case \cite{Hoggar1982,CalderbankHardinRainsShorSloane1999,Appleby2007,BachocEhler2013}.
For a larger number of subspaces, such equiangular arrangements cannot be realized and one needs to find an alternative bound
for the characterization of optimal packings, for example the orthoplex bound  for lines or subspaces \cite{ConwayHardinSloane1996}.
In an earlier paper, we constructed optimal line packings when the number of lines goes slightly beyond the threshold beyond which equiangular
lines are impossible to realize \cite{BodmannHaas}.
In this paper, we study the orthoplex bound for subspace packings and investigate cases in which the bound is achieved while the number of subspaces is maximal.

The main results are as follows. In order to maximize the number of subspaces while achieving the orthoplex bound, the 
 dimension of the subspaces is necessarily half of the dimension of the ambient space and the
chordal distance between subspaces assumes only two values.  Because of the relation with the orthoplex bound, we call these subspace packings maximal orthoplectic fusion frames.
The family of examples we describe here has already appeared in the literature, either as optimal real subspace packings \cite{ShorSloane1998}, whose discovery is ascribed to a ``remarkable coincidence'', or among the more general family of quantum 2-designs \cite{Zauner1999} in complex Hilbert spaces of prime power dimensions. In the complex case, it was observed that the projections are affine, Grassmannian designs \cite{Zauner1999}, see also \cite{Roy2010}, where the construction is attributed to R{\"o}tteler. 
In the present paper, we examine rigidity properties in the
construction of maximal orthoplectic fusion frames obtained from the theory of packings and designs \cite{Zauner1999, Bachoc2004,Roy2010,BachocEhler2013}.
We treat the real and complex case on the same footing, involving a new construction principle. To this end, we leverage earlier constructions of orthoplex-bound achieving, optimal line packings associated with mutually unbiased bases introduced by Schwinger \cite{Schwinger1960}. Maximal sets of mutually unbiased bases are known to exist in the complex case in prime power dimensions \cite{WoottersFields1989,BoykinSitharamTiepWocjan2007,GodsilRoy2009,Appleby2009} and in the real case if the dimension is a power of four \cite{CameronSeidel1973,LeCompteMartinOwens2010}, see also \cite{BoykinSitharamTarafiWocjan}. We obtain maximal orthoplectic fusion frames  
by augmenting these maximal mutually unbiased bases with block designs, subsets of the index set that satisfy certain combinatorial conditions. The designs we construct
for our purposes are known as balanced incomplete block designs and at the same time associated with optimal constant-weight binary codes \cite{Johnson1972}, see also
\cite{Sidelnikov1975}.
The resulting families of subspaces are constructed in any real 
Hilbert space whose dimension is a power of four or in any complex Hilbert space whose dimension is a power of two.

This paper is organized as follows. After the introduction, we fix notation and recall known distance and cardinality bounds on fusion frames in
Section~\ref{sec:bounds}. We relate the orthoplex bound for fusion frames with the notion of mutual unbiasedness in Section~\ref{sec:muff} and study implications
for the structure of maximal orthoplectic fusion frames as packings and as Grassmannian designs.
Finally, Section~\ref{sec:moff} presents the construction of a family of maximal orthoplectic fusion frames.

\section{Distance bounds and Grassmannian fusion frames}\label{sec:bounds}

\begin{defn}
Let $l,m,n \in \mathbb N$ and let $\BB{F}$ denote the field $\mathbb R$ or $\mathbb C$.  An {\bf {\boldmath$(n,l,m)$}-fusion frame} is a set $\CAL{F}=\{ P_j\}_{j=1}^n$, where each $P_j$ is an orthogonal projection onto an $l$-dimensional subspace of $\BB{F}^m$, such that there exist positive numbers $A$ and $B$ with $0<A\leq B$ for which the chain of inequalities
$$
A\|x\|^2 \leq \sum_{j=1}^n \|P_j x\|^2 \leq B\|x\|^2
$$
holds for every $x \in \BB{F}^m$.  If we can choose $A=B$, then $\CAL{F}$ is {\bf tight.} If there is $C\ge 0$ such that $\TR( P_i P_j ) = C$ for each pair $i \ne j$
in the index set $\ZN \equiv \{1, 2, \dots, n\}$, then
$\CAL F$ is called {\bf equiangular}.  
\end{defn}

By the polarization identity, the tightness property is equivalent to the fusion frame resolving the identity $I_m$ on $\BB{F}^m$ according to
$$
  \frac 1 A  \sum_{j=1}^n P_j = I_m \, .
$$
More general types of fusion frames are obtained by relaxing the condition that all subspaces have the same dimension and by scaling the projections with non-negative weight factors.

%
%


For any two projections $P$ and $P'$ onto $l$-dimensional subspaces of $\BB{F}^m$, the {\it chordal distance}
is defined by 
$
   d_c(P,P') = \frac{1}{\sqrt 2} \| P - P'\| = (l - \TR( P P' ) )^{1/2} \, . 
$
In order to 
characterize optimal packings with respect to $d_c$, we use an embedding that maps the projections to vectors in a higher dimensional Hilbert space.
We denote the dimension of this space as
 $$d_\BB{F} (m) = \left\{ \begin{array}{cc} \frac{(m+2)(m-1)}{2}, & \BB{F}= \BB{R} \\  m^2 -1, & \BB{F} =\BB{C} \end{array} \right. .$$

\begin{thm}[\cite{ConwayHardinSloane1996}] \label{thm:embed}
If $\CAL{F} = \{P_j\}_{j=1}^n$ is an $(n,l,m)$-fusion frame, then letting
$$
   V_j =   \sqrt{\frac{m}{l(m-l)}} \PARENTH{P_j - \frac{l}{m} I_m }
$$
defines a set of unit-norm vectors $\{V_j \}_{j=1}^n$ in the $\DF$-dimensional real Euclidean space
of symmetric/Hermitian $m \times m$ matrices with vanishing trace, equipped with the Hilbert-Schmidt norm, such that the inner products satisfy
 %
%
$$
\TR\PARENTH{P_j P_k}= \frac{l^2}{m} + \frac{l (m-l)}{m} \TR( V_j  V_k ),
$$
for every $j,k \in \ZN$.  Furthermore, if $\mathcal F$ forms a tight fusion frame for $\BB{F}^m$, then 
$
\sum_{j=1}^n V_j =0.
$
\end{thm}

%

We use Rankin's distance bound for vectors on the sphere \cite{Rankin1955} in the formulation used by Conway, Hardin and Sloane \cite{ConwayHardinSloane1996}.

\begin{thm}
Let $d$ be a positive integer.
Any $n$ vectors $\{v_1, v_2, \dots, v_n\}$ 
on the unit sphere in $\BB{R}^d$ have a minimum Euclidean distance
$
   \min_{j, k \in \ZN, j \ne k} \|v_j - v_k \| \le \sqrt{\frac{2n}{n-1}} \, ,
$
and if equality is achieved, then $n \le d+1$ and the vectors form a simplex.
Additionally, if $n>d+1$, then the minimum Euclidean distance is
$ \min_{j, k \in \ZN, j \ne k} \|v_j - v_k \| \le \sqrt{2},$
and if equality holds in this case,
then $n \le 2d$. Moreover, if $n=2d$, then equality holds if and only if the vectors form an orthoplex,
the union of an orthonormal basis with the negatives of its basis vectors.  
\end{thm}


In terms of the inner products between $n$ unit vectors in $\BB{R}^d$, the Rankin bound is
$$
  \max_{j, k \in \ZN, j \ne k} \langle v_j, v_k \rangle \ge - \frac{1}{n-1}, \, 
$$ 
and if $n>d+1$, then it improves to
$$
\max_{j, k \in \ZN, j \ne k} \langle v_j , v_k \rangle \ge 0 \, .
$$

Using the embedding from Theorem~\ref{thm:embed}, we reformulate the Rankin bound for the Hilbert-Schmidt inner products of the projections of a fusion frame.
This results in a bound that has already been derived in an alternative way before \cite{KutyniokPezeshkiCalderbankLiu2009,MasseyRuizStojanoff2010}
and in an improved bound for a larger number of subspaces, as noted in \cite{Roy2010}.

\begin{cor}\label{cor_2df}
If $\CAL{F}=\{P_j\}_{j=1}^n$ is a $(n,l,m)$-fusion frame,  
then 
$$
    \max\limits_{j,k \in\ZN, j \neq k} \TR(P_j P_k) \geq \frac{nl^2-ml}{m(n-1)}
$$
and if equality is achieved then the fusion frame is equiangular and $n \le \DF+1$.
If  $n >  \DF+1$, then
$$
\max\limits_{j,k \in\ZN, j \neq k} \TR(P_j P_k) \geq \frac{l^2}{m}
$$
and if equality is achieved then $n \le 2 \DF$. Moreover, if $n=2 \DF$, then
equality in this bound implies that for each $j \in \ZM$, $\TR(P_j P_k) =\frac{2(2l-m)}{m}$ for exactly one $k \in  \ZM \setminus \{j\}$  and $\TR(P_j P_k) = \frac{l^2}{m}$ for all other $k \in \ZM \setminus \{j\}$.
\end{cor}

\begin{defn}
Let $\CAL{F} =\CBRACK{P_j}_{j=1}^n$ be an $(n,l,m)$-fusion frame.   If $\CAL{F}$ is a solution to the subspace packing problem,
that is, it minimizes $\max_{j \ne k} \TR( P_j P_k)$ among all $(n,l,m)$-fusion frames, then it is called a {\bf Grassmannian}  fusion frame.
If $\CAL{F}$ is Grassmannian, then it is called {\bf orthoplex-bound achieving} if $n > \DF+1$ and 
$$
\max\limits_{j,k \in\ZN, j \neq k} \TR(P_j P_k) = \frac{l^2}{m}.
$$
If $\CAL{F}$ is orthoplex-bound achieving and $n=2 \DF$, then $\CAL{F}$ is referred to as a {\bf maximal orthoplectic fusion frame}.

\end{defn}

We wish to construct maximal orthoplectic fusion frames, which means the projections must embed exhaustively into the vertices of an orthoplex in $\BB{R}^{\DF}$. 
The feasability of this depends on whether the embedding admits antipodal points in the higher dimensional 
Euclidean sphere, which in turn depends on the relationship between $l$ and $m$.

\begin{prop}\label{prop_need_half_ss}
If $\CAL{F}=\{P_j\}_{j=1}^n$ is an $(n,l,m)$-fusion frame and $\CAL{V} = \{V_j\}_{j=1}^n \subset \BB{R}^{\DF}$ are the embedded vectors obtained from Theorem~\ref{thm:embed}, 
then
$$
\min\limits_{j,k \in\ZN, j \neq k}\langle V_j, V_k \rangle \geq -\frac{l}{m-l}
$$
for every $j,k \in \ZN$.  In particular, the embedding admits antipodal points only if $l \geq \frac{m}{2}$.
\end{prop}

\begin{proof}
This follows immediately from the inequality
$$
0 \leq \min\limits_{j,k \in\ZN, j \neq k} \TR\PARENTH{P_j P_k}= \frac{l^2}{m} + \frac{l (m-l)}{m} \min\limits_{j,k \in\ZN, j \neq k} \langle V_j, V_k\rangle.
$$
\end{proof}

Partitioning a maximal orthoplectic fusion frame into orthogonal pairs of projections shows, together with the preceding proposition,
that $m=2l$. 

\begin{cor} \label{cor:half_ss}
If $\mathcal F=\{P_j\}_{j=1}^n$ is a maximal orthoplectic $(n,l,m)$-fusion frame,
then $m=2l$.
\end{cor}

Maximal orthoplectic fusion frames enjoy another property that has been studied in the literature: they are part of a family of 
{\it Grassmannian 2-designs}, as shown by Zauner \cite{Zauner1999}. For our purposes, this is important because it imposes more rigidity
on their construction.
We follow Zauner's convention for the definition of these designs.

\begin{defn}
An $(n,l,m)$-fusion frame $\CAL{F}=\{P_j\}_{j=1}^n$ is a {\bf Grassmannian $t$-design} if
$$
   \sum_{j=1}^n  P_j^{\otimes t}  =     \sum_{j=1}^n  (U P_j U^*)^{\otimes t} 
$$
for any orthogonal matrix or unitary $U$ in the real or complex case, respectively,
where $P_j^{\otimes t} $ is the $t$-fold Kronecker/tensor product of $P_j$ with itself and $U^*$ is the adjoint of $U$ or the
transpose of $U$ in the real case.
\end{defn} 

 Equivalently, the 
 right-hand side of the defining identity can be averaged with respect to the Haar measure
$\mu$ on the group $\CAL U$ of orthogonal or unitary $m \times m$ matrices. This formulation implies a simple characterization
of the design property based on the {\it $t$-coherence tensor}
$$
  K_{t,l,m} = \int_{\CAL{U}} (U P U^*)^{\otimes t}  d\mu(U) \, ,
$$
where $P$ is any rank-$l$ orthogonal  projection matrix. Because of the analogy
with bounds for constant-weight codes \cite{Sidelnikov1975},
Zauner calls the following estimate a generalized Sidelnikov inequality.

\begin{thm}[Zauner, Theorem 2.5 of \cite{Zauner1999} ]\label{thm:ZaunerSidelnikov}
Let $\CAL{F}=\{P_j\}_{j=1}^n$ be an
$(n,l,m)$-fusion frame, then 
$$
   \frac{1}{n^2} \sum_{i,j=1}^n (\TR(P_i P_j))^t \ge \TR (K_{t,l,m}^2)
$$
and equality holds if and only if $\CAL F$
is a Grassmannian $t$-design.
 \end{thm}
 \begin{proof}
 Let $C= \sum_{j=1}^n  P_j^{ \otimes t } - n K_{t,l,m}$, then
 $\TR( C^2) = \sum_{i,j=1}^n (\TR(P_i P_j))^t - n^2 \TR( K_{t,l,m}^2 ) \ge 0$
 and cases of equality are characterized by $C=0$, which is the Grassmannian $t$-design property.
 \end{proof}

Because it is of independent interest, we study the $t$-coherence tensor and derive an alternative
proof of the Grassmannian 2-design property of maximal orthoplectic fusion frames.

 \begin{prop}\label{prop_K2lm}With respect to a fixed orthonormal basis $\{e_j\}_{j=1}^m$ for $\BB F^m$, the $2$-coherence tensor for an $(n,l,m)$-fusion frame can be
 reduced to the case $l=1$ by
 $$
    K_{2,l,m} = \PARENTH{ l - \frac{l(l-1)}{m-1}} K_{2,1,m} + \frac{l(l-1)}{m(m-1)} I_m \otimes I_m,
 $$
 and for $l=1$ it is given by
 $$
    K_{2,1,m}  = a \sum_{j=1}^m E_{j,j} \otimes E_{j,j} + \sum_{\substack{j,j'=1\\ j \ne j'}}^m \PARENTH{b (E_{j,j'} \otimes E_{j',j} + E_{j,j} \otimes E_{j',j'}) + c E_{j,j'} \otimes E_{j,j'}},
 $$
where we abbreviate $E_{j,j'} = e_j \otimes e_{j'}^*$ and 
 $a= \frac{\DF+ (m-1)^2}{m^2 \DF}$, $b=c=a/3$ if $\BB F = \BB R$ or $b=a/2$, $c=0$ if $\BB F = \BB C$.
  \end{prop}
 \begin{proof}
 In the special case $l=1$, by the normalization $\TR( K_{2,1,m}) = 1$ and the invariance properties of $K_{2,1,m}$
 under the tensor representation of the orthogonal or unitary group \cite[Ch.\ 7-9]{Ma2007},
we obtain 
 $$
    K_{2,1,m} = a \sum_{j=1}^m E_{j,j} \otimes E_{j,j} + \sum_{\substack{j,j'=1\\ j \ne j'}}^m \PARENTH{ b (E_{j,j'} \otimes E_{j',j} + E_{j,j} \otimes E_{j',j'}) + c E_{j,j'}  \otimes E_{j,j'}}.
 $$   
 
 Next, we use the symmetrization identity \cite{Zauner1999,RoyScott2007}
 \begin{align*}
   I_m \otimes I_m & = \sum_{j,j'=1}^m \int_{\CAL U} U E_{j,j} U^* \otimes U E_{j',j'} U^* d\mu(U) \\
      & = m K_{2,1,m} + m(m-1) \int_{\CAL U} U E_{1,1} U^* \otimes U E_{2,2} U^* d\mu(U) \, .
 \end{align*}
 Repeating this type of expansion for $K_{2,l,m}$ results in
 \begin{align*}
    K_{2,l,m} & = \sum_{j,j'=1}^l \int_{\CAL U} U E_{j,j} U^* \otimes U E_{j',j'} U^* d\mu(U)\\
    & = l K_{2,1,m} + l(l-1) \int_{\CAL U} U E_{1,1} U^* \otimes U E_{2,2} U^* d\mu(U) \\
    & =  l K_{2,1,m} + \frac{l(l-1)}{m(m-1)} ( I_m \otimes I_m - m K_{2,1,m} ) \, .
 \end{align*}
Rearranging terms gives the claimed expression for $K_{2,l,m}$.
 \end{proof}
 
 %
 %
 %
Next, we use the linear relation between the eigenvalues of $K_{2,l,m}$ and $K_{2,1,m}$ to obtain as
 a corollary the expression for the squared Hilbert-Schmidt norm of the $2$-coherence tensor for general $l \in \mathbb N$,
 which had been computed by Zauner without determining $K_{2,l,m}$ explicitly \cite[Lemma 2.7]{Zauner1999}.
 \begin{cor}\label{cor_K2lm_val}
Given positive integers $l$ and $m$ with $l \leq m$, the squared Hilbert-Schmidt norm of the $2$-coherence tensor is
 $$\displaystyle
    \TR(K_{2,l,m}^2) = \frac{l^4}{m^2} + \frac{l^2(m-l)^2}{\DF m^2}.
    %
 $$
 \end{cor}
 \begin{proof}
Using Proposition~\ref{prop_K2lm}, a straightforward computation for the case $l=1$ shows that, for both the real and complex case, we have
 $$
    \TR( K_{2,1,m}^2 ) = \frac{\DF+ (m-1)^2}{m^2 \DF}.
 $$

 Next, the identity relating $K_{2,l,m}$ and $K_{2,1,m}$  from Proposition~\ref{prop_K2lm} gives
 $$
    \TR( K_{2,l,m}^2 ) = \PARENTH{l - \frac{l(l-1)}{m-1}}^2 \TR( K_{2,1,m}^2 ) + 2 \PARENTH{l - \frac{l(l-1)}{m-1}} \frac{l(l-1)}{m(m-1)} + \frac{l^2 (l-1)^2}{(m-1)^2} \, .
 $$
 Inserting the value of the squared Hilbert-Schmidt norm for $K_{2,1,m}$ and simplifying this expression gives
 $$
    \TR( K_{2,l,m}^2 ) = \frac{1}{\DF m^2 (m-1)} \PARENTH{\DF l^4(m-1) + (m-1) l^2 (m-l)^2} = \frac{l^4}{m^2} + \frac{l^2(m-l)^2}{\DF m^2} \, .
 $$
 \end{proof}

Next, we verify that maximal orthoplectic fusion frames are  Grassmannian $2$-designs.



\begin{thm} Given a maximal orthoplectic $(2\DF,m/2,m)$-fusion frame $\CAL{F} = \{P_j \}_{j=1}^{2\DF}$, then equality holds in the generalized Sidelnikov inequality 
in Theorem~\ref{thm:ZaunerSidelnikov}
and
the fusion frame is a Grassmannian 2-design.
\end{thm}
\begin{proof}
We first compute the value of the lower bound from Theorem~\ref{thm:ZaunerSidelnikov},
$$
   \TR ( K_{2,m/2,m}^2 ) = \frac{m^2(\DF + 1)}{16 \DF} \, .
$$ 
Next, we use Corollary~\ref{cor_2df} to compute the average squared inner product for the projection matrices,
$$
   \frac{1}{4\DF^2} \sum_{j,j'=1}^{2\DF} (\TR( P_j P_{j'} ) )^2 = \frac{1}{2\DF} \PARENTH{ l^2 + \frac{l(2l-m)}{m} + (2\DF -2) \frac{l^4}{m^2} } = \TR ( K_{2,m/2,m}^2 ),
$$
where the last equality follows by $l=m/2$. As stated in Theorem~\ref{thm:ZaunerSidelnikov}, this characterizes Grassmannian 2-designs.
\end{proof}

\section{Mutually unbiased bases and fusion frames}\label{sec:muff}


\begin{defn}
If $\CAL{B}=\{b_j\}_{j=1}^m$ and $\CAL{B'}=\{b'_j\}_{j=1}^m$ are a pair of orthonormal bases for $\BB{F}^m$, then they are {\bf mutually unbiased} if
$$
|\langle b_j, b'_{j'} \rangle|^2 = \frac{1}{m}
$$
for every $j,j' \in\ZM$. A collection of orthonormal bases $\{\CAL{B}_k\}_{k \in K}$ is called a set of {\bf mutually unbiased bases} if the pair $\CAL{B}_k$ and $\CAL{B}_{k'}$
is mutually unbiased for every $k \ne k'$.
\end{defn}

The number of mutually unbiased bases is bounded in terms of the dimension $m$.

\begin{thm}[Delsarte, Goethals and Seidel \cite{DelsarteGoethalsSeidel1975}] \label{th_mub_bd}\label{thm:DGS}
Let $\BB{F}^m$ be a real or complex Hilbert space and let
$\{\CAL{B}_k \}_{k \in K}$ be a set of mutually unbiased bases for $\BB{F}^m$, where $\CAL{B}_k = \CBRACK {b_j^{(k)} }$ for each $k \in K$, and let $r=|K|$.  If $\BB{F} = \BB R$, then $r \leq m/2+1$
and if $\BB F = \BB C$, then $r \le m+1$.
 If equality is achieved in either case, then the real span of 
the corresponding projection operators,  $ \CBRACK{b_j^{(k)} \otimes \PARENTH{b_j^{(k)}}^* : j\in\ZM, k \in K}$, 
is that of
all symmetric or Hermitian operators on $\BB{F}^m$, respectively.
\end{thm}
\begin{proof}
After selecting an appropriate ordering on the vectors, 
the Gram matrix  of the corresponding rank-one projection operators
 has the form
$G=I_{r}\otimes I_m+ \frac{1}{m}(J_{r}-I_{r})\otimes J_m,$ where 
$I_{r}$ and $I_m$ are the $r \times r$ and $m\times m$ identity matrices, respectively, and 
$J_{r}$ and $J_m$ denote the $r\times r$ and $m\times m$ matrices containing only $1$'s.
The kernel of the Gram matrix is the space of vectors $a\otimes b$ such that
$J_m b= m b$ and $J_{r} a=0$, so it is $\PARENTH{r-1}$-dimensional. Consequently, the rank of $G$
and thus the real dimension of the span of the rank one projections  is $rm-r+1$.
This shows the claimed bound on $r$ and that the maximal rank is achieved if and only if
$r =m+1$ in the complex case or
$r=\frac{m}{2}+1$ in the real case, because the rank of the Gram matrix equals the dimension of the span of the underlying projections.
\end{proof}

For the real case, it is known that, for most values of $m$, the maximal number of mutually unbiased bases is less than or equal to $3$ \cite{LeCompteMartinOwens2010}; however,
if $m$ is a power of $4$, then examples exist that achieve the bound in Theorem~\ref{th_mub_bd} \cite{CameronSeidel1973}. 
In the complex case, the bound is achieved if $m$ is a prime power \cite{WoottersFields1989}.

\begin{thm}[Cameron and Seidel \cite{CameronSeidel1973}, Wootters and Fields \cite{WoottersFields1989}]\label{th_prime_mubs_exist}
If $m$ is a prime power, then a family of $m+1$ mutually unbiased bases for $\BB{C}^m$ exists.
If $m$ is a power of 4, then a family of $m/2+1$ mutually unbiased bases exist for $\BB{R}^m$.
\end{thm}

Henceforth, we abbreviate $k_{\mathbb R}(m) = m/2+1$ and $k_{\mathbb C}(m) = m+1$.
The rank one orthogonal projections corresponding to maximal sets of mutually unbiased bases give rise to  Grassmannian 2-designs.
 
 \begin{prop}\label{lem_MUBs_2tight}
If $\{\CAL{B}_k\}_{k \in K}$ is a set of mutually unbiased bases for $\BB{F}^m$ with $|K| = \KF$, then 
the family of rank-one projections
$\CAL{F} =  \CBRACK{b_j^{(k)} \otimes \PARENTH{b_j^{(k)}}^* : k \in K, j\in\ZM}$ forms a 
Grassmannian 2-design.
\end{prop}
\begin{proof}
We only need to compare both sides of the inequality from Theorem~\ref{thm:ZaunerSidelnikov}.
To evaluate the left-hand side, we observe that the Hilbert-Schmidt inner product
is expressed in terms of the basis vectors as $\PARENTH{\TR \PARENTH{ P_j^{(k)} P_{j'}^{(k')} }}^2 = \left| \langle b_j^{(k)}, b_{j'}^{(k')}\rangle \right|^4$. 
%
Given any fixed basis vector $b_{j}^{(k)}$, the fourth power of the absolute value of its inner product with the other vectors in the set  are $0$, which occurs $m-1$ times, $1/m^2$, which occurs $(\KF-1)m$ times, and $1$, which occurs once. 
Averaging these gives
$$
  \frac{1}{m^2 \KF^2}  \sum_{j,j'=1}^m \sum_{k,k'=1}^{\KF}  \PARENTH{\TR\PARENTH{P^{(k)}_j P^{(k')}_{j'}}}^2 = \frac{\KF+ m-1}{m^2 \KF} \, .
$$
Comparing with the value of $\TR( K_{2,l,m}^2 )$ in the special case $l=1$ and using $\DF = (m-1)\KF $ shows that
equality holds in the inequality in Theorem~\ref{thm:ZaunerSidelnikov}.
 \end{proof}

The version of the orthoplex bound for projections motivates the notion of mutual unbiasedness for fusion frames.

\begin{defn}
If $\CAL{F}=\{P_j\}_{j=1}^n$ is an  $(n,l,m)$-fusion frame and $\CAL{F'}=\{P'_j\}_{j=1}^{n'}$ is an $(n',l,m)$-fusion frame, then $\CAL F$ and $\CAL F'$ are {\bf mutually unbiased} if
$$
\TR(P_j P'_{j'}) = \frac{l^2}{m}
$$
for every $j \in\ZN$ and  $j'\in\ZNprime$. 
 A collection of fusion frames 
$\CBRACK{\CAL F_k}_{k \in K}$ for
 $\BB F^m$, where each $\CAL F_k$ consists of projections onto $l$-dimensional subspaces, is a set of 
{\bf mutually unbiased fusion frames} if the pair $\CAL F_k$ and $\CAL F_{k'}$ is mutually unbiased for every $k \neq k'$. 
\end{defn}

Given a subset of a fixed orthonormal basis, the orthogonal projection onto the span is given by the sum of the corresponding rank one projections.  Projections formed in this way are called {\em coordinate projections.}

\begin{defn}
Given an orthonormal basis $\CAL{B} = \{b_j\}_{j=1}^m$ for $\BB{F}^m$ and a subset $\CAL J \subset \ZM$, the {\bf $\CAL J$-coordinate projection} with respect to $\CAL{B}$ is
$$P_{\CAL J} =\sum\limits_{j \in \CAL J} b_j \otimes b_j^*.$$
\end{defn}

Given a pair of mutually unbiased bases, then one can select  coordinate projections from the respective bases to form mutually unbiased fusion frames.

\begin{prop}\label{prop_MUBs_give_MUFFs}
If $\CAL{B} = \CBRACK{b_j}_{j=1}^m$ and ${\CAL{B}'}=\CBRACK{{b'}_j}_{j=1}^{m}$ are a pair of mutually unbiased bases for $\BB{F}^m$ and $\CAL J, \CAL{J}' \subset \ZM$ with $l=|\CAL J|=|\CAL J'|$, then
$
\TR\PARENTH{P_{\CAL J} {P'}_{\CAL J'}} = \frac{l^2}{m},$
where
$P_{\CAL J}$ is the $\CAL J$-coordinate projection with respect to $\CAL{B}$ and  ${P'}_{\CAL J'}$ is $\CAL J'$-coordinate projection with respect to ${\CAL{B}'}.$  Moreover, if $\CAL{F} = \{ P_{\CAL J}\}_{\CAL J \in  \BLOCKSET}$ is a set of coordinate projections with respect to $\CAL{B}$, $\CAL{F}' = \{ {P'}_{\CAL J'}\}_{\CAL J' \in  \BLOCKSET'}$  is a set of coordinate projections with respect to $\CAL{B}'$, $\CAL{F}$ is an  $(\left| \BLOCKSET \right|,l,m)$-fusion frame and $\CAL{F}'$ is a $(\left| \BLOCKSET'\right|,l,m)$-fusion frame, then $\CAL{F} \cup \CAL{F}'$ is mutually unbiased.
\end{prop}

\begin{proof}
We compute
$$
\TR(P_{\CAL J} {P'}_{\CAL J'}) 
= \sum\limits_{j \in \CAL J, j \in \CAL J'} \TR
\PARENTH{
b_j \otimes \PARENTH{b_j}^* {b'}_{j'} \otimes \PARENTH{ {b'}_{j'}}^* 
}
=  \sum\limits_{j \in \CAL J, j' \in \CAL J'} \left|\left\langle b_j, {b'}_{j'} \right\rangle \right|^2= \frac{l^2}{m}
.
$$
The claim about mutual unbiasedness follows directly from this computation.
\end{proof}

We repeat the embedding of fusion frames for the special case of coordinate projections. Tight fusion frames of coordinate projections have also
been investigated as commutative quantum designs by Zauner \cite{Zauner1999}. We first focus on the structure of optimal packings of coordinate projections.

\begin{thm} \label{thm:coordembed}
Let $\CAL{B} = \{b_j\}_{j=1}^m$ be an orthonormal basis for $\BB{F}^m$ and let $\BLOCKSET = \{\CAL{J}_j\}_{j=1}^n$ be a set of subsets of $\ZM$, each of size $|\CAL{J}_j|=l$.  If $\CAL{F}=\{P_j\}_{j=1}^n$ is a family of projections 
for which  $P_j$ is the $\CAL{J}_j$-coordinate projection with respect to $\CAL{B}$,
then the set of unit vectors $\CAL{V} =\{V_j\}_{j=1}^n$
obtained as in Theorem~\ref{thm:embed} resides in a $m-1$-dimensional subspace
of the real Euclidean space of symmetric matrices and
$$
\TR\PARENTH{P_j P_{j'}}= \frac{l^2}{m} + \frac{l (m-l)}{m} \langle V_j, V_{j'}\rangle,
$$
for every $j,j' \in \ZN$.  
\end{thm}

\begin{proof}
By definition,  $\CAL{F}$ is a set of rank-$l$ orthogonal projections that can be regarded as diagonal matrices when represented in the basis $\CAL{B}$.
The mapping
$
P_j \mapsto  V_j := \sqrt{\frac{m}{l(m-l)}} (P_j -  \frac{l}{m} I_m ) $
embeds the projections into the real diagonal matrices with zero trace. This implies
$
\dim\bigl(\SPAN \CBRACK{V_j}_{j=1}^n   \bigr) \leq m-1 .
$
The identity for the Hilbert-Schmidt inner products follows directly from the definition of $\{V_j\}_{j=1}^n$.  
\end{proof}
In this special case, the Rankin bound can be expressed in terms of the subsets indexing the coordinate projections,
because $\TR( P_{\CAL{J}} P_{\CAL{J}'} ) = |\CAL{J} \cap \CAL{J}'|$ for any  $\CAL J, \CAL{J}' \subset \ZM$. 
A more general bound of this type has been derived by Johnson in the context of {\it constant-weight codes} \cite[Inequality (14)]{Johnson1972}.

\begin{cor} \label{cor:1designbd}
 If $\BLOCKSET$ is a collection of $n$ subsets of $\ZM$ for which $n>m$ and each $\CAL J \in \BLOCKSET$
 has size $|\CAL J|=l$, then
 $$
    \max_{\CAL J , \CAL{J}' \in \BLOCKSET, \CAL J \ne \CAL{J}'} |\CAL J \cap \CAL{J}' | \ge \frac{l^2}{m} \, .
 $$
 Moreover, if $n=2(m-1)$ and equality holds in this bound, then $\BLOCKSET$ can be partitioned into $m-1$ disjoint pairs of 
 subsets of size $m/2$.
\end{cor}

In light of Corollary~\ref{cor:1designbd} and Proposition~\ref{prop_MUBs_give_MUFFs}, we pursue the  idea of generating tight, orthoplex-bound achieving fusion frames by using coordinate projections from {\it maximal sets of mutually unbiased bases}, which are sets of mutually unbiased bases that achieve the cardinality bound in Theorem~\ref{thm:DGS}.
In order to construct an orthoplex-bound achieving fusion frame, 
 we need 
$n > \DF +1$ subspaces.  Thus, given a maximal set of mutually unbiased bases, then we need a sufficient number of coordinate projections per basis with low Hilbert-Schmidt inner products.  
In order to bound the inner products between coordinate projections corresponding to a given orthonormal basis, the 
intersection of any two different index sets $\CAL J$ and $\CAL{J}'$ 
must have a small intersection, which we call a cohesiveness bound.
According to Corollary~\ref{cor:1designbd}, the maximum number of such subsets whose intersections are at most of 
size $l^2/m$ is $2(m-1)$.

\begin{defn}
Let $ \BLOCKSET$ be a collection of subsets of $\ZM$, each $\CAL{J} \in \BLOCKSET$ of size $l$.
We say that $\BLOCKSET$ is {\bf $c$-cohesive} if there exists $c>0$ such that
$$
\max\limits_{\substack{\CAL J, \CAL{J}' \in  \BLOCKSET \\ \CAL J \neq \CAL{J}' }} \left| \CAL J \cap \CAL{J}' \right| \leq c.
$$
If $\BLOCKSET$ is $l^2/m$-cohesive and $| \BLOCKSET| = 2(m-1)$, then it is {\bf maximally orthoplectic}.\end{defn}

\begin{thm}\label{th_costruct_OGFF}
Let $ \BLOCKSET$ be an $l^2/m$-cohesive collection of subsets of $\ZM$, where each $\CAL{J} \in \BLOCKSET$ is of size $l$, let $\{ \CAL{B}_k \}_{k \in K}$ be a set of mutually unbiased bases for $\BB{F}^m$, where $|K| |  \BLOCKSET| > \DF+1$ and $\CAL{B}_k = \CBRACK{ b_j^{(k)} }_{j=1}^m$ for each $k \in K$, and let $n =  | K | | \BLOCKSET|$.  If
the set 
$
\CAL{F} = \CBRACK{ P_{\CAL J}^{(k)} : k \in K, \CAL J \in  \BLOCKSET}
$
forms an $(n,l,m)$-fusion frame, where each
 $P_{\CAL J}^{(k)}$ denotes the $\CAL J$-coordinate projection with respect to $\CAL{B}_k$, then 
$\CAL F$ is an orthoplex-bound achieving  $(n,l,m)$-fusion frame.
Moreover, if $(\ZM,  \BLOCKSET)$ is maximally orthoplectic and if $|K|= \KF$, then the set $\CAL{F}$ is a tight, maximal orthoplectic  fusion frame.
\end{thm}

\begin{proof}
The cardinality requirement in the definition of orthoplex-bound achieving fusion frames is satisfied since $n > \DF +1$.
Let  $k, k' \in K$.
If $k \neq k'$, then 
$$
\TR\PARENTH{ P_{\CAL J}^{(k)} P_{\CAL{J}'}^{(k')}   } = \frac{l^2}{m}
$$
for every $\CAL J, \CAL{J}' \in  \BLOCKSET$ by 
Proposition~\ref{prop_MUBs_give_MUFFs}.
If $k = k'$, then the fact that $\BLOCKSET$ is  $l^2/m$-cohesive  yields
\begin{align*}
\max\limits_{\substack{ \CAL J, \CAL{J}' \in  \BLOCKSET \\ \CAL J \neq \CAL{J}'}}\TR\PARENTH{ P_{\CAL{J}}^{(k)} P_{\CAL{J}'}^{(k)}   }
&=
\max\limits_{\substack{ \CAL J, \CAL{J}' \in  \BLOCKSET \\ \CAL J \neq \CAL{J}'}}
\left| \CAL J \cap \CAL{J}' \right|
\leq
\frac{l^2}{m},
 \end{align*}
which shows that $\CAL{F}$ is an orthoplex-bound achieving  fusion frame.
Finally, if $\BLOCKSET$ is maximally orthoplectic and the set of mutually unbiased bases is maximal, then 
Corollary~\ref{cor:1designbd} shows that the coordinate projections belonging to each basis sum to a multiple of the identity, so the corresponding fusion frame is tight.
Hence, the union of all the coordinate projections belonging to the mutually unbiased bases 
forms a set of $n = \KF \left(2 m-2\right)= 2\DF$ orthogonal projections whose pairwise inner products 
are bounded by $l^2/m$, which shows that $\CAL{F}$ is a maximal orthoplectic  fusion frame.
\end{proof}

Following Zauner's ideas, we repeat the study of design properties for the special case of a fusion frame formed by coordinate projections.
To this end, we define the {\it diagonal coherence tensor}, 
$$
    D_{t,l,m} = \frac{1}{\left({m \atop l }\right)} \sum_{\CAL J \in \BB J} D_{\CAL J}^{\otimes t} \, ,
$$  where $\BB J$ is the set of all subsets of $\ZM$ of size $l$,
and, for each $\CAL{J} \in \BB J$,  $D_{\CAL J}$ is the $\CAL J$-coordinate projection with respect to the canonical basis.
An elementary counting argument shows $D_{1,l,m} = \frac{l}{m} I_m$ and
$$
   D_{2,l,m} = \frac{l}{m}\sum_{j=1}^m E_{j,j} \otimes E_{j,j} + \frac{l(l-1)}{m(m-1)} \sum_{\substack{j,j'=1\\j \ne j'}}^m E_{j,j} \otimes E_{j',j'}, \, 
$$
where $\{e_j\}_{j=1}^m$ denotes the canonical basis for $\BB F^m$ and $E_{j,j'}=e_j \otimes e_{j'}^*$ for each $j,j' \in \ZM$.
By squaring the diagonal entries and summing, we compute
$$
    \TR (D_{2,l,m}^2 ) = m \frac{l^2}{m^2} + m(m-1) \frac{l^2(l-1)^2}{m^2 (m-1)^2} = \frac{l^2}{m (m-1)}(l^2 - 2l + m) \, .
$$

With this notation, the combinatorial notion of a {\it block $t$-design}  is characterized conveniently.
\begin{defn}
A {\bf $t$-$(m,l, \lambda)$ block design} $\BLOCKSET$ is a  collection of subsets of $\ZM$ called {\bf blocks}, where each block $\CAL{J} \in \BB S$ has cardinality $l$, 
such that every subset of $\ZM$ with cardinality $t$ is contained in exactly $\lambda$ blocks.
When the parameters are not important or implied by the context, then $\BLOCKSET$ is also referred to as a {\bf block $t$-design}.
The special case of a block $2$-design is also referred to as a {\bf balanced incomplete block design}.
\end{defn}

\begin{thm}[Zauner {\cite[Theorem 1.12]{Zauner1999}}]\label{thm:diagdesigns}
A collection $\BLOCKSET$ of subsets of $\ZM$, where each $\CAL J  \in \BLOCKSET$ has size $l$, is a $t$-$(m,l, \lambda)$ block design if and only if
$$
   \frac{1}{|\BLOCKSET|} \sum_{\CAL J \in \BLOCKSET} D_{\CAL J}^{\otimes t} = D_{t,l,m}, \, 
$$
with
$
 \lambda = |\BLOCKSET | \TR \PARENTH{ D_{t,l,m} \bigotimes_{s=1}^t E_{s,s}}, 
$ 
where $\{e_j\}_{j=1}^m$ denotes the canonical basis for $\BB F^m$ and $E_{s,s}=e_s \otimes e_{s}^*$ for each $s \in \ZM$.
\end{thm}
\begin{proof}
The definition of $\BB{J}$ implies that any subset of $\ZM$ of size $t$ is contained in a fixed number of sets from $\BB J$.
Since both sides of the claimed identity are diagonal in the standard basis, the block design property is a consequence of
the fact that for any subset $\{j_1, j_2, \dots, j_t\} \subset \ZM$, 
$P_{\CAL J}^{\otimes t}$ has an  eigenvector $ e_{j_1} \otimes e_{j_2} \otimes \cdots \otimes e_{j_t} $
corresponding to eigenvalue one 
if and only if $\{j_1, j_2, \dots, j_t\} \subset \CAL J$; otherwise, it corresponds to eigenvalue zero.  The claimed value for $\lambda$ follows from an elementary counting argument.
\end{proof}

In the special case where $t=1$ in Theorem~\ref{thm:diagdesigns}, 
we obtain the correspondence between the block 1-design property of $\BLOCKSET$ and tightness of the corresponding fusion frame of coordinate projections.

\begin{cor}\label{cor:prop_1designs}
If $\CAL{B}$ is any orthonormal basis for $\BB{F}^m$, then 
a set of coordinate projections $\{P_{\CAL J}\}_{\CAL J \in \BLOCKSET}$ with respect to $\CAL B$ is a tight fusion frame if and only if $\BLOCKSET$
is a {\it $1$-design}. 
\end{cor}

Given any positive integers $l$ and $m$ with $l \leq m$, one can choose the set of all blocks of size $l$ from $\ZM$ to form a tight fusion frame in this way.  

\begin{ex}
If $ \BLOCKSET = \CBRACK{\CAL  J : \CAL J \subset \ZM, |\CAL J|=l }$, the set of all blocks of size $l$,  then $\BLOCKSET$ forms a $t$-$(m,l, \lambda)$ block design.
Given
an orthonormal basis $\CAL{B}$ for $\BB F^m$, then the corresponding set of coordinate projections with respect to $\CAL{B}$, $\CAL{F} = \{P_{\CAL J}\}_{\CAL J \in  \BLOCKSET}$, forms a tight $(n,l,m)$-fusion frame by Corollary~\ref{cor:prop_1designs}, where $n =\left|  \BLOCKSET \right| = \PARENTH{{ m \atop l }}$.
\end{ex}

With the same proof as in Theorem~\ref{thm:ZaunerSidelnikov}, we obtain an analogous  characterization of block $t$-designs.
\begin{cor} \label{cor:diagdesigns}
Given a collection $\BLOCKSET$ of subsets of $\ZM$, where each $\CAL J  \in \BLOCKSET$ has size $l$, then
$$
  \frac{1}{|\BLOCKSET|^2} \sum_{\CAL J, \CAL{J}' \in \BLOCKSET} | \CAL J \cap \CAL{J}' |^t \ge \TR ( D_{t,l,m}^2 )
$$
and equality holds if and only if $ \BLOCKSET$ is a $t$-$(m,l,\lambda)$ block design with
$
 \lambda 
$
as in Theorem~\ref{thm:diagdesigns}.
\end{cor}

We can now deduce that subsets of coordinate projections
in
a maximal orthoplectic fusion frame constructed from a
maximal set of mutually unbiased bases realize block $2$-designs. We state this more generally as a correspondence between
Grassmannian 2-designs and balanced incomplete block designs.

\begin{thm} \label{thm:GdesignBIBD}
Let $\{\CAL{B}_k \}_{k \in K}$ be a maximal set of mutually unbiased bases for $\BB F^m$, so $|K|=\KF$ and let $\BLOCKSET \subset \ZM$ be a collection of subsets, each with size $l$.
If $\CAL{F} = \CBRACK{ P_{\CAL J}^{(k)}: k \in K, \CAL J \in \BLOCKSET }$ is a $(2\DF, m/2, m)$-fusion frame, where each $P^{(k)}_{\CAL J}$ is  the $\CAL J$-coordinate projection with respect to $\CAL{B}_k$,
then $\CAL{F}$ is a Grassmannian 2-design if and only if $\BLOCKSET$ is a $2$-$(m,m/2,m/2-1)$ block design.
\end{thm}

\begin{proof}
First, let $\CAL{F}$ be a Grassmannian 2-design. By Corollary~\ref{cor_K2lm_val} and the choice of $l=m/2$,
$$
   \sum_{k,k' \in K}\sum_{\CAL J,\CAL J' \in \BLOCKSET} \PARENTH{ \TR(P^{(k)}_{\CAL J} P^{(k')}_{\CAL J'} )}^2 = \frac{m^2 \DF (\DF +1)}{4} \, .
$$
From the assumption on the size $\KF = |K|$ and $\DF=(m-1) \KF$, the set $\BLOCKSET$ has size $|\BLOCKSET|=2(m-1)$.
Since the orthormal bases are unitarily equivalent, and each pair of them is mutually unbiased, we can obtain the sum
for the squared Hilbert-Schmidt inner products belonging to one basis,
\begin{align*}
  \sum_{\CAL J,\CAL J' \in \BLOCKSET} \PARENTH{\TR(P^{(k)}_{\CAL J} P^{(k)}_{\CAL J'} )}^2 & =  \frac{m^2 \DF (\DF+1)}{4 \KF} -  (\KF - 1)  4 (m-1)^2 \frac{m^2}{16}\, \\
  & = \frac{(m-1) m^2}{4} \bigl(\KF (m-1) + 1 -  (\KF - 1)(m-1)\bigr) = \frac{(m-1) m^3}{4} \, .
\end{align*}
The average of the Hilbert-Schmidt inner products of the $|\BLOCKSET|=2(m-1)$ coordinate projections
belonging to each basis is then
\begin{align*}
    \frac{1}{|\BLOCKSET|^2} \sum_{\CAL J,\CAL J' \in \BLOCKSET} \PARENTH{ \TR( P^{(k)}_{\CAL J} P^{(k)}_{\CAL J'} ) }^2 & =  \frac{1}{|\BLOCKSET|} \PARENTH{l^2 + (n-2) \frac{l^4}{m^2} } 
     =  \frac{m^3}{16(m-1)} \, .
\end{align*}
Specializing the expression  $ \TR \PARENTH{D_{2,l,m}^2 }= l^2(l^2-2l+m)/m(m-1)$ to $l=m/2$ shows that equality holds in Corollary~\ref{thm:GdesignBIBD},
so $\BLOCKSET$ is a block $2$-design. The parameter of the design then follows from $\lambda = |\BLOCKSET| \frac{l(l-1)}{m(m-1)} =m/2-1$. 

Conversely, if  $\BLOCKSET$ is a $2$-$(m,m/2,m/2-1)$ block design,
then equality holds in the inequality in Corollary~\ref{cor:diagdesigns}.
Since the squared inner product between any two coordinate projections belonging to different bases equals $l^4/m^2$, the lower bound
from Corollary~\ref{cor:diagdesigns} is equivalent to a lower bound for the squared inner products among the coordinate projections
belonging to all mutually unbiased bases, and both bounds are saturated. Using the preceding two identities shows that 
this implies that equality holds in the inequality in Theorem~\ref{thm:ZaunerSidelnikov}
and hence $\CAL{F}$ is a Grassmannian 2-design. 
\end{proof}




%

\section{A family of maximal orthoplectic  fusion frames}\label{sec:moff}
In this section, we construct a family of  $\{0,1\}$-matrices, $\{S_r\}_{r \in \BB{N}}$, and show that they generate maximally orthoplectic block $1$-designs, and therefore generate maximal orthoplectic fusion frames by Theorem~\ref{th_costruct_OGFF}. Consequently, by Theorem~\ref{thm:GdesignBIBD} they are 2-designs. We give an independent
proof of this fact to illustrate the rigidity in the construction of these matrices.

We recall from  Corollary~\ref{cor:half_ss} that a necessary condition for the existence of a maximal orthoplectic fusion frame is that the subspace dimension is
 $l =\frac{m}{2}$.  In order to exploit the existence of maximal sets of mutually unbiased bases in prime power dimensions, it is natural to focus on the case where $m$ is a power of two. We construct the block $1$-designs in terms of the associated  {\it incidence matrices}.
 
 \begin{defn}
 The {\bf incidence matrix} $S$ associated with a sequence   $ \BLOCKSET=\{\CAL J_1$, $\CAL J_2, \dots, \CAL J_n\}$ of subsets of $\ZM$
 is an $m\times n$ matrix whose $(a,b)$-th entry is
$$
S_{a,b} = \left\{ \begin{array}{cc} 1, &  a \in \CAL J_b \\ 0, & \text{otherwise}\end{array} \right. 
.$$
 \end{defn}

Let $S_1 =I_2$.  For $r \in \BB{N}$, let  $F_r = I_{(2^{r}-1)} \otimes  \left( \begin{array}{cc} 0 & 1 \\ 1 & 0\end{array} \right)$ and let $1_r$ denote the $2^{r} \times 1$ matrix of all ones and let $0_r$ be the $2^{r} \times 1$ matrix of all zeros.  For $r \geq 2$,
define
$S_r$ recursively and block-wise by 
$$
S_r = \left(B_r^{(i)} \,\, B_r^{(ii)}  \, \, B_r^{(iii)}\right), 
$$
where
$$
B_r^{(i)} = \left( \begin{array}{cc} 1_{r-1} & 0_{r-1}\\                                                
                                                    0_{r-1} & 1_{r-1} 
         \end{array} \right),
B_r^{(ii)} = \left( \begin{array}{cc} S_{r-1}\\
                                                   S_{r-1}
         \end{array} \right),
\text{ and }
B_r^{(iii)} = \left( \begin{array}{cc} S_{r-1}\\
                                                    S_{r-1} F_{r-1}
         \end{array} \right).
$$

If $c_r$ and $\rho_r$ denote the number of columns and rows of $S_t$, respectively, then we have the recurrence relation
$$
c_1=2, c_{r+1} = 2 c_{r} +2, 
$$
which has the solution $c_r = 2^{r+1}-2$.  By the construction of $S_r$, $\rho_{r+1} =2 \rho_r = 2^r$, so $S_r$ is a $(2^{r+1}-2) \times 2^r$ matrix.

Furthermore, if $\tilde c^{(j)}_r$ denotes the number of ones in the $j$th column of $S_r$ and  $\tilde \rho^{(j)}_r$ denotes the number of ones in the $j$th row of $S_r$, then it is straightforward to verify, by construction, that both of these values are independent of $j$.  In particular, $\tilde c_r^{(j)} = 2^{r-1}$ for each $j$ and $\tilde \rho^{(j)}_r= 2^{r}-1$ for each $j$ in the index set of columns or rows, respectively.
  We record this as a lemma.
\begin{lemma}\label{lem_halfones}
Each column of $S_r$ has exactly $2^{r-1}$ ones among its entries, and each row of $S_r$ has exactly $2^r -1$ ones among its entries.
\end{lemma}

Next, we examine the inner products among the columns, $\{s_j\}_{j \in \ZCT}$, of $S_r$, noting that these are encoded in the matrix,
$$
S_r^* S_r = \PARENTH{\langle s_b, s_a \rangle}_{a,b =1}^{c_r}.
$$
We write ${J}_{x,y}$ for the $x\times y$ matrix whose entries all equal $1$.

\begin{lemma}\label{lem_st_gram}
For each $r \in \mathbb{N}$, the matrix $G=S_r^* S_r$ is of the form
$$
G= 2^{r-2} \cdot \left[{J}_{c_r ,c_r} + {I}_{c_r/2} \otimes \left( \begin{array}{cc}1 & -1 \\ -1 & 1 \end{array}  \right)  \right].
$$
\end{lemma}

\begin{proof}
We prove the claimed form of $G$ by induction.  The claim is true for $r=1$, so let $r>1$ assume the claim holds for $r-1$.

Using the block structure in the definition of $S_r$, we have
\begin{align*}
{\tiny
G = 
\left(
\begin{array}{ccc}
\PARENTH{B_r^{(i)} }^* B_r^{(i)}  &   \PARENTH{B_r^{(i)} }^* B_r^{(ii)}  & \PARENTH{B_r^{(i)} }^* B_r^{(iii)}
\\
\PARENTH{B_r^{(ii)} }^* B_r^{(i)}  &   \PARENTH{B_r^{(ii)} }^* B_r^{(ii)}  & \PARENTH{B_r^{(ii)} }^* B_r^{(iii)}
\\
\PARENTH{B_r^{(iii)} }^* B_r^{(i)}  &   \PARENTH{B_r^{(iii)} }^* B_r^{(ii)}  & \PARENTH{B_r^{(iii)} }^* B_r^{(iii)}
\end{array}
\right).
}
\end{align*}
A direct application of the definition of $B_r^{(i)}$ and Lemma~\ref{lem_halfones} gives the values of the first row and first column of blocks in $G$,
$${\tiny G=
\left(
\begin{array}{ccc}
2^{ r-1} \cdot I_2  &  2^{r-2} \cdot J_{2 , c_{r-1} }& 2^{r-2} \cdot J_{2 ,c_{r-1}}
\\
2^{r-2} \cdot J_{ c_{r-1} , 2} &   \PARENTH{B_r^{(ii)} }^* B_r^{(ii)}  & \PARENTH{B_r^{(ii)} }^* B_r^{(iii)}
\\
2^{r-2} \cdot J_{ c_{r-1} , 2} &   \PARENTH{B_r^{(iii)} }^* B_r^{(ii)}  & \PARENTH{B_r^{(iii)} }^* B_r^{(iii)}
\end{array}
\right)}.$$
A direct application of the induction hypothesis gives us the center block,
$$
  \PARENTH{B_r^{(ii)} }^* B_r^{(ii)} = S_{r-1}^* S_{r-1}  + S_{r-1}^* S_{r-1}  = 2^{r-2} \cdot \left[J_{c_{r-1} , c_{r-1}} + I_{(c_{r-1}/2)} \otimes \left( \begin{array}{cc}1 &-1 \\ -1 &1 \end{array}  \right)  \right].
$$
Next, observe that by the induction assumption and definition of $F_{r-1}$, we have
\begin{align*}
S_{r-1}^* S_{r-1}
F_{r-1} &=
2^{r-3} \cdot \left[J_{c_{r-1} ,c_{r-1}} + I_{c_r/2} \otimes \left( \begin{array}{cc}1 & -1 \\ -1 & 1 \end{array}  \right)  \right]
\SBRACK{
I_{c_{r-1}/2} \otimes  \left( \begin{array}{cc} 0 & 1 \\ 1 & 0\end{array} \right)
}
\\
&=
2^{r-3} \cdot \left[J_{c_{r-1} , c_{r-1}} + I_{c_r/2} \otimes \left( \begin{array}{cc}-1 & 1 \\ 1& -1 \end{array}  \right)  \right],
\end{align*}
so it follows that
$$
 \PARENTH{B_r^{(ii)} }^* B_r^{(iii)} 
=
  S_{r-1}^* S_{r-1}  + S_{r-1}^* S_{r-1}  F_{r-1} \\
= 
 2^{r-2} \cdot J_{c_{r-1} , c_{r-1}},
$$
and, by symmetry, we also have
$$
 \PARENTH{B_r^{(iii)} }^* B_r^{(ii)} 
= 
 2^{r-2} \cdot J_{c_{r-1} ,c_{r-1}}.
$$
Finally, observe that $F_{r-1}^* S_{r-1}^* S_{r-1} F_{r-1} =  S_{r-1}^* S_{r-1}$, so
\begin{align*}
 \PARENTH{B_r^{(iii)} }^* B_r^{(iii)} 
&= S_{r-1}^* S_{r-1}  +F_{r-1}^* S_{r-1}^* S_{r-1} F_{r-1} 
\\
&= 2^{r-2} \cdot \left[J_{c_{r-1} , c_{r-1}} + I_{c_{r-1}/2} \otimes \left( \begin{array}{cc}1 &-1 \\ -1 &1 \end{array}  \right)  \right].
\end{align*}
This establishes that the nine blocks match the claimed form of $G$. 
\end{proof}


For each $r \in \BB N$, we let $\BLOCKSET_r$ be  the set of blocks in $\ZM$ defined in accordance with the columns of $S_r$ by
$
\CAL{J}_b = \CBRACK{a: (S_r)_{a,b}=1},
$
where $(S_r)_{a,b}$ denotes the $(a,b)$-th entry of $S_t$.

Although the block 2-design property of $\BLOCKSET_r$ is implicit in the result on maximal orthoplectic fusion frames stated below, we show it in a separate
proof to illustrate the additional constraints realized by the construction.

\begin{prop}
For each $r \in \BB{N}$, $S_r$ 
is the incidence matrix of a $2$-$(m,m/2,m/2-1)$ block design, with $m=2^r$.
\end{prop}
\begin{proof}
Since the block $2$-design property of $\BLOCKSET_r$ is equivalent to the statement that every row of $S_r$ has constant sum and any two distinct row vectors have a constant inner product, it is sufficient to prove the matrix identity
$$
   S_ rS_r^* = 2^{r-1} I_{2^r} + (2^{r-1}-1)J_{2^r,2^r} \, . 
$$

We prove this by induction jointly with an ancillary claim,
$$
  S_r F_r S_r^* = 2^{r-1}J_{2^r,2^r} -2^{r-1} I_{2^r}\, .
$$
To begin,
$S_1$ satisfies $S_1 S_1^* = I_2$ and
 $S_1 F_1 S_1 = F_1$, so both identities hold for $r=1$.
 
 Assuming this is also true for $S_{r-1}$ and $F_{r-1}$, we compute
 $$
     S_r S_r^* = \left( \begin{array}{cc} J_{2^{r-1},2^{r-1}} + 2 S_{r-1} S_{r-1}^* & S_{r-1} S_{r-1}^*+ S_{r-1} F_{r-1} S_{r-1}^*\\
                 S_{r-1} S_{r-1}^*+  S_{r-1} F_{r-1} S_{r-1}^* & J_{2^{r-1},2^{r-1}} + S_{r-1} S_{r-1}^* + S_{r-1} F_{r-1}^2 S_{r-1}^*
                 \end{array} \right).
 $$
 Now using the induction assumption and $F_{r-1}^2 = I_{2^{r-1}}$, we get
 $$
    S_r S_r^* = 2(2^{r-2}) I_{2^r} + 2(2^{r-2}-1)J_{2^r,2^r} + J_{2^r,2^r},
    $$
    which simplifies to the claimed identity.
    Moreover, using the fact that $F_r$ has the block form,
        $$F_r = \PARENTH{\begin{array}{cc} F_1 & 0_{2 , 2 c_{r-1} } \\  0_{ 2 c_{r-1} , 2 }& I_2 \otimes F_{r-1} \end{array}},$$ where $0_{a,b}$ is the $a \times b$ zero matrix, 
a straightforward computation yields
    $$
        S_r F_r S_r^* = \left( \begin{array}{cc}  2 S_{r-1} F_{r-1} S_{r-1}^* & J_{2^{r-1},2^{r-1}}+ S_{r-1} F_{r-1} S_{r-1}^*+ S_{r-1} S_{r-1}^*\\
                 J_{2^{r-1},2^{r-1}}+ S_{r-1} F_{r-1} S_{r-1}^*+  S_{r-1} S_{r-1}^* & 2 S_{r-1} F_{r-1} S_{r-1}^* \end{array} \right)
    $$ 
    and using the induction assumption gives
    $$
      S_r F_r S_r^* = 2^{r-1} \left( \begin{array}{cc}   J_{2^{r-1},2^{r-1}} -  I_{2^{r-1}} &  J_{2^{r-1},2^{r-1}}\\
                  J_{2^{r-1},2^{r-1}} &  J_{2^{r-1},2^{r-1}} -  I_{2^{r-1}} \end{array} \right) = 2^{r-1} (J_{2^r,2^r} -  I_{2^r}) \, .
    $$
\end{proof}

Finally, we state the main theorem of this section, which summarizes the construction of maximal orthoplectic tight fusion frames.  

\begin{thm}\label{th_main}
Let $r \in \BB{N}$, where $r$ is even if $\BB{F}=\BB{R}$. If $m=2^r$ and  $\{\CAL B_k\}_{k \in K}$ is a maximal collection of mutually unbiased bases
for for $\BB{F}^m$, so $|K|=\KF$, then
$$
\CAL{F} = 
\CBRACK{
P_{\CAL{J}}^{(k)} : k \in K,  \CAL{J} \in \BLOCKSET_{r}
}
$$
forms a tight, maximal orthoplectic fusion frame,
where each $P_\CAL{J}^{(k)}$ is the $\CAL{J}$-coordinate projection with respect to $\CAL{B}_k$.
\end{thm}
\begin{proof}
 It follows directly from Lemma~\ref{lem_halfones} that $\BLOCKSET_r$ is a $1$-$(m,l, m-1)$ block design, where $m=2^r$, $l=2^{r-1}$ and $| \BLOCKSET_r| = c_r = 2^{r+1}-2$.  
If $\{s_j\}_{j\in \ZCT}$ denotes the columns of $S_r$, then the Gram matrix $S_r^* S_r$ encodes the intersections of the blocks by $\langle s_a,s_b\rangle~= ~\left| \CAL{J}_a \cap \CAL{J}_b\right|$, so Lemma~\ref{lem_st_gram} implies that $\max\limits_{a \ne b} \left| \CAL{J}_a \cap \CAL{J}_{b} \right| = 2^{r-2}$.  This means that $\BLOCKSET_r$ is $c$-cohesive, where $c=\frac{l^2}{m} = \frac{m}{4}$, and  since $| \BLOCKSET_r| = 2^{r+1}-2 = 2(m-1)$, we conclude that $\BLOCKSET_r$ is a maximally orthoplectic block $1$-design.

Finally, using a maximal set of mutually unbiased bases $\{\CAL B_k\}_{k \in K}$ and the
 maximally orthoplectic block $1$-$(m,m/2,m-1)$ design  $\BLOCKSET_r$, Theorem~\ref{th_costruct_OGFF} shows that the set
$
\CAL{F} = 
\{
P_{\CAL{J}}^{(k)} : k \in K$, $\CAL{J} \in \BLOCKSET_{r}
\}
$
forms a maximal orthoplectic fusion frame,
where each $P_\CAL{J}^{(k)}$ is the $\CAL{J}$-coordinate projection with respect to $\CAL{B}_k$.
\end{proof}

\end{document}